\documentclass{birkjour}
 \newtheorem{thm}{Theorem}[section]
 \newtheorem{cor}[thm]{Corollary}
 \newtheorem{lem}[thm]{Lemma}
 
 \theoremstyle{definition}

 \theoremstyle{remark}
 \newtheorem{rem}[thm]{Remark}
 \numberwithin{equation}{section}
\usepackage{color}
\begin{document}

%
%
%
%
%
%
%
%
%

\title[On some equalities and inequalities for $K$-frames]
 {On some equalities and inequalities for $K$-frames}



\author[F. Arabyani Neyshaburi]{Fahimeh Arabyani Neyshaburi}
\address{Department of Mathematics and Computer Sciences, Hakim Sabzevari University, Sabzevar, Iran.}
\email{arabianif@yahoo.com}

\author[Gh. Mohajeri Minaei]{Ghadir Mohajeri Minaei}
\address{Technical and Vocational University, Neyshabur Branch, Iran.}
\email{gh-minaei1391@yahoo.com}

\author[E. Anjidani]{Ehsan Anjidani}
\address{Department of Mathematics, University of  Neyshabur, P.O.Box 91136-899, Neyshabur, Iran.}
\email{anjidani@neyshabur.ac.ir}



\vspace{1.7cm}
\begin{abstract}
$K$-frame theory  was recently introduced   to reconstruct elements from the range of a bounded linear operator $K$ in a separable Hilbert space. This significant property  is worthwhile especially in some problems arising in sampling theory. Some equalities and inequalities have been established for ordinary frames and their duals. In this paper, we   continue and extend  these results  to obtain several  important equalities and inequalities for $K$-frames. Moreover,  by applying Jensen's operator inequality  we obtain some new inequalities for $K$-frames.
\end{abstract}

\maketitle
\maketitle
\section{Introduction and preliminaries}

\smallskip
\goodbreak
Frames are redundant systems  in separable Hilbert spaces, which  provide  non-unique representations of vectors and  are applied in a wide range of applications \cite{Ben06, Bod05, Bol98, Cas00}.   Working on efficient algorithms for computing reconstruction of  signals without  noisy phase, Balan et al. \cite{Balan2}  discover some new identities for Parseval frames. Then, these identities have been generalized by some researchers for the alternate dual frames and  other types of frames \cite{Balan, Gav06, Xiao2, Zhu}.

Recently, $K$-frames were introduced  by G$\breve{\textrm{a}}$vru\c{t}a \cite{Gav07}  to study atomic systems with respect to a bounded operator $K\in B(\mathcal{H})$.  Recall that atomic decomposition for a closed subspace $\mathcal{H}_{0}$ of a Hilbert space $\mathcal{H}$ introduced by  Feichtinger et al.  in \cite{Han}
 with frame-like properties. Although, the sequences in  atomic decompositions
do not necessarily  belong to $\mathcal{H}_{0}$.  This interesting  property, which was the main motivation  of introducing atomic decomposition and $K$-frame theory, comes from some problems   in sampling theory   \cite{Paw, Werthrt}. In this work, we  extend and improve  these results  to obtain some  equalities and inequalities for $K$-frames.  First, we are going to state some preliminaries of $K$-frames and their duals, which are used in our main results.
Let  $\mathcal{H}$ be a separable Hilbert space and $I$  a countable index set and $K\in B(\mathcal{H})$.  A sequence $F:=\lbrace f_{i}\rbrace_{i \in I} \subseteq \mathcal{H}$ is called a $K$-frame for $\mathcal{H}$, if there exist constants $A, B > 0$ such that
\begin{eqnarray}\label{001}
A \Vert K^{*}f\Vert^{2} \leq \sum_{i\in I} \vert \langle f,f_{i}\rangle\vert^{2} \leq B \Vert f\Vert^{2}, \quad (f\in \mathcal{H}).
\end{eqnarray}
Obviously, $F$ is an ordinary frame  if $K=I_{\mathcal{H}}$  and so $K$-frames are a generalization of the ordinary frames  \cite{Cas00, Chr08, Gav07}. The constants $A$ and $B$ in $(\ref{001})$ are called the lower and the upper bounds of $F$, respectively. A $K$-frame  $F$ is called a Parseval $K$-frame, whenever $\sum_{i\in I} \vert \langle f,f_{i}\rangle\vert^{2} = \Vert K^{*}f\Vert^{2}$. Since $F$ is a Bessel sequence,   similar to ordinary frames the synthesis operator can be defined as $T_{F}: l^{2}\rightarrow \mathcal{H}$; $T_{F}(\{ c_{i}\}_{i\in I}) = \sum_{i\in I} c_{i}f_{i}$. The operator $T_{F}$ is  bounded  and its adjoint which is called the analysis operator is given by $T_{F}^{*}(f)= \{ \langle f,f_{i}\rangle\}_{i\in I}$, and the frame operator is given by $S_{F}: \mathcal{H} \rightarrow \mathcal{H}$; $S_{F}f = T_{F}T_{F}^{*}f = \sum_{i\in I}\langle f,f_{i}\rangle f_{i}$. Unlike ordinary frames, the frame operator of a K-frame is not invertible in general. However,  if $K$ has closed range, then $S_{F}$ from $R(K)$ onto $S_{F}(R(K))$ is an invertible operator \cite{Xiao}.

In \cite{arefi3}, the notion of duality for $K$-frames is introduced and several approaches for construction and characterization of $K$-frames and their dual are presented. Indeed, a Bessel sequence
 $\{g_{i}\}_{i \in I}\subseteq \mathcal{H}$ is called a \textit{$K$-dual} of $\{ f_{i} \}_{i\in I}$ if
\begin{eqnarray}\label{dual1}
Kf = \sum_{i\in I} \langle f,g_{i}\rangle f_{i}, \quad (f\in \mathcal{H}).
\end{eqnarray}

\begin{thm}[Douglas   \cite{Douglas}]\label{equ0}
Let $L_{1}\in B(\mathcal{H}_{1}, \mathcal{H})$ and $L_{2}\in B(\mathcal{H}_{2}, \mathcal{H})$ be bounded linear mappings on given Hilbert spaces. Then the following assertions are equivalent:
\begin{itemize}
\item[(i)]
$R(L_{1}) \subseteq R(L_{2})$;
\item[(ii)]
$L_{1}L_{1}^{*} \leq \lambda^{2}L_{2}L_{2}^{*}$, \quad for some $\lambda > 0$;
\item[(iii)]
There exists a bounded linear mapping $X\in L(\mathcal{H}_{1}, \mathcal{H}_{2})$, such that $L_{1} = L_{2}X$.
\end{itemize}
Moreover, if (i), (ii) and (iii) are valid, then there exists a unique operator $X$ so that $L_{1}=L_{2}X$
\begin{itemize}
\item[(a)]
$\Vert X\Vert ^{2} = \inf \{\alpha>0, L_{1}L_{1}^{*}\leq \alpha L_{2}L_{2}^{*}\}$;
\item[(b)]
$N( L_{1}) = N(X)$;
\item[(c)]
$R(X) \subset \overline{R( L_{2}^{*})}$.
\end{itemize}
\end{thm}
\begin{rem}\label{First rem}
Suppose  $F = \{f_{i}\}_{i\in I}$ is a $K$-frame  for $\mathcal{H}$ with the optimal bounds $A$ and $B$, respectively. Clearly $B=\Vert S_{F}\Vert$. Also,    using   Douglas' theorem,   there exists an operator $X\in B(\mathcal{H}, l^{2})$ such that
\begin{eqnarray}\label{01.}
 T_{F}X=K,
\end{eqnarray}
 and $\{X^{*}\delta_{i}\}_{i\in I}$ is a $K$-dual of $F$, where $\{\delta_{i}\}_{i\in I}$ is the standard orthonormal basis of $l^{2}$, see \cite{Gav07}.
Moreover, by Douglas' theorem the equation (\ref{01.})  has a unique solution as $X_{F}$ such that
\begin{eqnarray}\label{XW}
\Vert X_{F}\Vert ^{2} = \inf \{\alpha>0,K K^{*}\leq \alpha T_{F}T_{F}^{*}\}.
\end{eqnarray}
Now, let $A$ be the  optimal lower bound of $F$. Then, we obtain
\begin{eqnarray*}
A &=& sup\{\alpha>0: \quad \alpha K K^{*}f \leq T_{F}T_{F}^{*}\}\\
&=&\left(inf\{\beta>0: \quad  K K^{*}\leq \beta T_{F}T_{F}^{*}\}\right)^{-1}\\
&=& \Vert X_{F}\Vert^{-2},
\end{eqnarray*}
which  coincides  with ordinary frames. Indeed, if  $K=I_{\mathcal{H}}$, we easily  obtain $X_{F}=T_{F}^{*}S_{F}^{-1}$ and so
 $\Vert X_{F}\Vert^{-2} =\Vert T_{F}^{*}S_{F}\Vert^{-2}=\Vert S_{F}^{-1}\Vert^{-1}$.
\end{rem}
 For further information in $K$-frame theory  we refer the reader  to  \cite{arefi3, Han, Gav07,  Xiao}.
Throughout this paper, we suppose  $\mathcal{H}$ is a separable Hilbert space, $K^{\dag}$ the pseudo inverse of operator $K$, $I$ a countable index set and for every $J\subset I$, we show the complement of $J$ by $J^{c}$. Also
$I_{\mathcal{H}}$ denotes  the identity operator on $\mathcal{H}$. For two Hilbert spaces $\mathcal{H}_{1}$ and $\mathcal{H}_{2}$ we denote by $B(\mathcal{H}_{1},\mathcal{H}_{2})$ the set of all bounded linear operators between $\mathcal{H}_{1}$ and $\mathcal{H}_{2}$, and we abbreviate $B(\mathcal{H},\mathcal{H})$ by $B(\mathcal{H})$. Also we denote the range of $K\in B(\mathcal{H})$ by $R(K)$, and the orthogonal projection of $\mathcal{H}$ onto a closed subspace $V$ of $\mathcal{H}$  by $\pi_{V}$. Finally, we use of $T_{J}$ and  $S_{J}$ to denote the synthesis operator and frame operator, whenever the index set is  limited to $J\subset I$.

\smallskip
\goodbreak
\section{Main Results}

In this section we obtain several  equalities and inequalities for $K$-frames, $K$-duals and Parseval $K$-frames. These results  extend and improve  some important results of \cite{Balan, Gav06, Zhu}.

\begin{thm}\label{1}
Let $F = \{f_{i}\}_{i\in I}$ be a $K$-frame for $\mathcal{H}$ and $\{g_{i}\}_{i\in I}$ be  a $K$-dual of $F$. Then for every $J\subseteq I$ and $f\in \mathcal{H}$,
\begin{eqnarray*}
\left(\sum_{i\in J}\langle f, g_{i}\rangle \overline{\langle Kf, f_{i}\rangle} \right)-  \left \Vert \sum_{i\in J}\langle f, g_{i}\rangle f_{i}\right \Vert^{2} = \left(\sum_{i\in J^{c}}\overline{\langle f, g_{i}}\rangle \langle Kf, f_{i}\rangle \right)- \left \Vert \sum_{i\in J^{c}}\langle f, g_{i}\rangle f_{i}\right \Vert^{2}.
\end{eqnarray*}
\end{thm}
\begin{proof}
Suppose that $J\subseteq I$, and consider the operator
\begin{eqnarray*}
M_{J}f = \sum_{i\in J}\langle f, g_{i}\rangle f_{i}, \quad (f\in \mathcal{H}).
\end{eqnarray*}
One can easily show that $M_{J}$ is a  well defined and bounded operator on $\mathcal{H}$. Moreover, we have $M_{J}+M_{J^{c}} = K$. Hence,
\begin{eqnarray*}
&&\left(\sum_{i\in J}\langle f, g_{i}\rangle \overline{\langle Kf, f_{i}\rangle} \right)-  \left \Vert \sum_{i\in J}\langle f, g_{i}\rangle f_{i}\right \Vert^{2}\\
&=& \langle K^{*}M_{J}f, f\rangle - \langle M_{J}^{*}M_{J}f, f\rangle\\
&=&  \langle (K^{*}-M_{J}^{*})M_{J}f, f\rangle \\
&=& \langle M_{J^{c}}^{*}(K-M_{J^{c}})f, f\rangle \\
&=& \langle M_{J^{c}}^{*}Kf, f\rangle - \langle M_{J^{c}}^{*}M_{J^{c}}f, f\rangle\\
&=& \langle f, K^{*}M_{J^{c}}f\rangle - \Vert M_{J^{c}}f\Vert ^{2}\\
&=& \left(\sum_{i\in J^{c}}\overline{\langle f, g_{i}}\rangle \langle Kf, f_{i}\rangle \right)- \left \Vert \sum_{i\in J^{c}}\langle f, g_{i}\rangle f_{i}\right \Vert^{2}.
\end{eqnarray*}
\end{proof}
Using Theorem \ref{1} we immediately obtain the following corollary.
\begin{cor}
Let $F = \{f_{i}\}_{i\in I}$ be a $K$-frame with a $K$-dual  $\{g_{i}\}_{i\in I}$. Then for every $J\subseteq I$ and $f\in \mathcal{H}$,
\begin{eqnarray*}
&&\langle T_{J}\{(X_{F}f)_{i}\}_{i\in J}, Kf \rangle- \Vert T_{J}\{(X_{F}f)_{i}\}_{i\in J}\Vert^{2}\\
&=& \overline{\langle T_{J^{c}}\{(X_{F}f)_{i}\}_{i\in J^{c}}, Kf\rangle} -  \Vert T_{J^{c}}\{(X_{F}f)_{i}\}_{i\in J^{c}}\Vert^{2},
\end{eqnarray*}
where $X_{F}$ is as in $(\ref{XW})$.
\end{cor}
\begin{thm}
Let $F = \{f_{i}\}_{i\in I}$ be a $K$-frame for $\mathcal{H}$ and $\{g_{i}\}_{i\in I}$ be a $K$-dual of $F$. Then for every bounded sequence $\{\alpha_{i}\}_{i\in I}$ and $f\in \mathcal{H}$,
\begin{eqnarray*}
&&\left(\sum_{i\in I}\alpha_{i} \langle f, g_{i}\rangle \overline{\langle Kf, f_{i}\rangle} \right)-  \left \Vert \sum_{i\in I}\alpha_{i}\langle f, g_{i}\rangle f_{i}\right \Vert^{2}\\
&=& \left(\sum_{i\in I}(1-\overline{\alpha_{i}})\overline{\langle f, g_{i}}\rangle \langle Kf, f_{i}\rangle \right)- \left \Vert \sum_{i\in I}(1-\alpha_{i})\langle f, g_{i}\rangle f_{i}\right \Vert^{2}.
\end{eqnarray*}
\end{thm}
In the sequel, we survey some inequalities on Parseval $K$-frames. These results also extend some important inequalites for frames in \cite{Balan, Gav06}.
\begin{thm}\label{parseval1212}
Assume that $F = \{f_{i}\}_{i\in I}$ is a Parseval $K$-frame for $\mathcal{H}$. For every  $f\in \mathcal{H}$, $J\subseteq I$ and  $E\subseteq J^{c}$ we have
\begin{eqnarray*}
&&\left \Vert \sum_{i\in J\cup E}\langle f, f_{i}\rangle f_{i}\right \Vert^{2}- \left \Vert \sum_{i\in J^{c}\setminus E}\langle f, f_{i}\rangle f_{i}\right \Vert^{2} \\
&=& \left \Vert \sum_{i\in J}\langle f, f_{i}\rangle f_{i}\right \Vert^{2}- \left \Vert \sum_{i\in J^{c}}\langle f, f_{i}\rangle f_{i}\right \Vert^{2}+2Re\sum_{i\in E}\langle f, f_{i}\rangle\overline{\langle KK^{*}f, f_{i}\rangle},
\end{eqnarray*}
\end{thm}
\begin{proof}
First note that for every $J\subseteq I$, we have $S_{J}+S_{J^{c}} = KK^{*}$. Hence $S_{J}^{2}-S_{J^{c}}^{2} =KK^{*}S_{J}-S_{J^{c}}KK^{*}$. In fact
\begin{eqnarray*}
S_{J}^{2}-S_{J^{c}}^{2} &=& S_{J}^{2} - (KK^{*}-S_{J})^{2}\\
&=& KK^{*}S_{J}+S_{J}KK^{*} - (KK^{*})^{2}\\
&=&  KK^{*}S_{J}  -(KK^{*}- S_{J}) KK^{*}\\
&=&KK^{*}S_{J}-S_{J^{c}}KK^{*}.
\end{eqnarray*}
Hence, for every $f\in \mathcal{H}$ we obtain
 \begin{eqnarray*}
\Vert S_{J}f\Vert^{2}-\Vert S_{J^{c}}f\Vert^{2}= \langle KK^{*}S_{J}f, f\rangle -\langle S_{J^{c}}KK^{*}f, f\rangle,
\end{eqnarray*} and consequently
\begin{eqnarray*}
&&\left \Vert \sum_{i\in J\cup E}\langle f, f_{i}\rangle f_{i}\right \Vert^{2}- \left \Vert \sum_{i\in J^{c}\setminus E}\langle f, f_{i}\rangle f_{i}\right \Vert^{2} \\
&=& \sum_{i\in J\cup E}\langle f, f_{i}\rangle \overline{\langle KK^{*}f, f_{i}\rangle}-\overline{\sum_{i\in J^{c}\setminus E}\langle f, f_{i}\rangle \overline{\langle KK^{*}f, f_{i}\rangle}}\\
&=&\sum_{i\in J}\langle f, f_{i}\rangle \overline{\langle KK^{*}f, f_{i}\rangle}-\overline{\sum_{i\in J^{c}}\langle f, f_{i}\rangle \overline{\langle KK^{*}f, f_{i}\rangle}}+2Re\sum_{i\in E}\langle f, f_{i}\rangle \overline{\langle KK^{*}f, f_{i}\rangle}\\
&=& \left \Vert \sum_{i\in J}\langle f, f_{i}\rangle f_{i}\right \Vert^{2}- \left \Vert \sum_{i\in J^{c}}\langle f, f_{i}\rangle f_{i}\right \Vert^{2}+2Re\sum_{i\in E}\langle f, f_{i}\rangle\overline{\langle KK^{*}f, f_{i}\rangle}.
\end{eqnarray*}
\end{proof}

\begin{thm}\label{10}
Let $F = \{f_{i}\}_{i\in I}$ be a Parseval $K$-frame for $\mathcal{H}$. Then for every $J\subseteq I$ and $f\in \mathcal{H}$,
\begin{eqnarray*}
&&Re\left(\sum_{i\in J^{c}}\langle f, f_{i}\rangle \overline{\langle KK^{*}f, f_{i}\rangle} \right)+ \left \Vert \sum_{i\in J}\langle f, f_{i}\rangle f_{i}\right \Vert^{2}\\
&=&Re\left(\sum_{i\in J}\langle f, f_{i}\rangle \overline{\langle KK^{*}f, f_{i}\rangle} \right)+ \left \Vert \sum_{i\in J^{c}}\langle f, f_{i}\rangle f_{i}\right \Vert^{2} \geq \dfrac{3}{4}\Vert KK^{*}f\Vert^{2}.
\end{eqnarray*}
\end{thm}
\begin{proof}
Since  $S_{J}^{2}-S_{J^{c}}^{2}= KK^{*}S_{J}-S_{J^{c}}KK^{*}$, we can write
\begin{eqnarray*}
S_{J}^{2}+S_{J^{c}}^{2} = 2\left(\dfrac{KK^{*}}{2} - S_{J}\right)^{2}+\dfrac{(KK^{*})^{2}}{2}\geq \dfrac{(KK^{*})^{2}}{2},
\end{eqnarray*}
and consequently
\begin{eqnarray*}
&&KK^{*}S_{J}+S_{J^{c}}^{2}+ (KK^{*}S_{J}+S_{J^{c}}^{2})^{*}\\
&=& KK^{*}S_{J}+S_{J^{c}}^{2}+S_JKK^*+S_{J^{c}}^{2}\\
&=& KK^{*}(S_{J}+S_{J^{c}}) + S_{J}^{2}+S_{J^{c}}^{2} \\
&=& (S_{J}+S_{J^{c}})KK^{*} + S_{J}^{2}+S_{J^{c}}^{2}
\geq \dfrac{3}{2}(KK^{*})^{2}.
\end{eqnarray*}
Thus, we obtain
\begin{eqnarray*}
&&Re\left(\sum_{i\in J^{c}}\langle f, f_{i}\rangle \overline{\langle KK^{*}f, f_{i}\rangle} \right)+ \left \Vert \sum_{i\in J}\langle f, f_{i}\rangle f_{i}\right \Vert^{2}\\
&=& Re\left(\sum_{i\in J}\langle f, f_{i}\rangle \overline{\langle KK^{*}f, f_{i}\rangle} \right)+ \left \Vert \sum_{i\in J^{c}}\langle f, f_{i}\rangle f_{i}\right \Vert^{2}\\
&=& \dfrac{1}{2}\left(\langle KK^{*}S_{J}f, f\rangle + \langle S_{J^{c}}^{2}f, f\rangle+ \langle f, KK^{*}S_{J}f\rangle + \langle f, S_{J^{c}}^{2}f\rangle \right)\\
&\geq& \dfrac{3}{4}\Vert KK^{*}f\Vert^{2}.
\end{eqnarray*}
This completes the proof.
\end{proof}

Notice that, in Theorem \ref{10} if we take $K=I_{\mathcal{H}}$ it reduce to Theorem 2.2 in \cite{Zhu}.
Also,  Theorem \ref{10} leads to the following concept, which is a generalization of \cite{Gav06} for Parseval frames. Let $F = \{f_{i}\}_{i\in I}$ be a Parseval $K$-frame. Then we consider
\begin{eqnarray*}
&&v_{+}(F, K, J)=\sup_{f\neq 0}\dfrac{Re\left(\sum_{i\in J^{c}}\overline{\langle f, f_{i}\rangle} \langle KK^{*}f, f_{i}\rangle \right)+ \left \Vert \sum_{i\in J}\langle f, f_{i}\rangle f_{i}\right \Vert^{2}}{\Vert KK^{*}f\Vert^{2}},\\
&&v_{-}(F, K, J)=\inf_{f\neq 0}\dfrac{Re\left(\sum_{i\in J^{c}}\overline{\langle f, f_{i}\rangle} \langle KK^{*}f, f_{i}\rangle \right)+ \left \Vert \sum_{i\in J}\langle f, f_{i}\rangle f_{i}\right \Vert^{2}}{\Vert KK^{*}f\Vert^{2}}.
\end{eqnarray*}
Now, we are going to  present some properties of these notations.  For this we need the next lemma.
\begin{lem}\label{1.5}
Let $K$ be a closed range operator and $F=\{f_{i}\}_{i\in I}$ be  a  $K$-frame for $\mathcal{H}$. Then
\begin{itemize}
\item[(i)]
$\left \Vert \sum_{i\in I}\langle f, f_{i}\rangle f_{i}\right \Vert^{2}\leq \Vert S_{F}\Vert \sum_{i\in I}\vert\langle f, f_{i}\rangle \vert^{2}, \quad (f\in \mathcal{H})$.
\item[(ii)]
$\sum_{i\in I}\vert\langle f, f_{i}\rangle \vert^{2} \leq \Vert K^{\dagger}\Vert^{2}\Vert X_{F}\Vert^{2} \left \Vert \sum_{i\in I}\langle f, f_{i}\rangle f_{i}\right \Vert^{2}, \quad (f\in R(K))$.
\end{itemize}
\end{lem}
\begin{proof}
The first part is  a known result for every Bessel sequence \cite{Chr08} and so for  $K$-frames. Hence, we only need to show  $(ii)$.
Let $f\in R(K)$. Then we have
\begin{eqnarray*}
\left(\sum_{i\in I}\vert\langle f, f_{i}\rangle \vert^{2}\right)^{2}&=&\vert \langle S_{F}f, f\rangle\vert^{2}\\
&\leq& \Vert S_{F}f\Vert^{2} \Vert f\Vert^{2}\\
&=& \Vert S_{F}f\Vert^{2} \Vert (K^{\dagger})^{*} K^{*}f\Vert^{2}\\
&\leq& \Vert S_{F}f\Vert^{2} \Vert (K^{\dagger})\Vert^{2} \Vert K^{*}f\Vert^{2}\\
&\leq& \Vert S_{F}f\Vert^{2} \Vert (K^{\dagger})\Vert^{2} \Vert X_{F}\Vert^{2} \sum_{i\in I}\vert\langle f, f_{i}\rangle \vert^{2}.
\end{eqnarray*}
\end{proof}
\begin{thm}
Suppose  $F = \{f_{i}\}_{i\in I}$ is a Parseval $K$-frame for $\mathcal{H}$. The following assertion hold:
\begin{itemize}
\item[(i)]
$\dfrac{3}{4} \leq v_{-}(F, K, J)\leq v_{+}(F, K, J)\leq \Vert K\Vert \Vert K^{\dagger}\Vert (1+\Vert K\Vert)$
\item[(ii)]
$ v_{+}(F, K, J)=  v_{+}(F, K, J^{c})$, and $ v_{-}(F, K, J)= v_{-}(F, K, J^{c})$
\item[(iii)]
$ v_{+}(F, K, I)= v_{-}(F, K, I)  = 1$, and $ v_{+}(F, K, \emptyset)= v_{-}(F, K, \emptyset) =1.$
\end{itemize}
\end{thm}
\begin{proof}
For (i), it is sufficient to show the upper inequality. Since $F$ is a Bessel sequence, so by applying Lemma  \ref{1.5} (i) we obtain
\begin{eqnarray*}
\left \Vert \sum_{i\in J}\langle f, f_{i}\rangle f_{i}\right \Vert^{2}&\leq& \Vert S_{J}\Vert \sum_{i\in J}\vert \langle f, f_{i}\rangle\vert^{2}\\
&\leq& \Vert S_{J}\Vert \sum_{i\in I}\vert \langle f, f_{i}\rangle\vert^{2}\\
&\leq& \Vert K\Vert^{2} \Vert K^{*}f\Vert^{2}\\
&=& \Vert K\Vert^{2} \Vert K^{\dagger}KK^{*}f\Vert^{2}\\
&\leq& \Vert K\Vert^{2} \Vert K^{\dagger}\Vert^{2} \Vert KK^{*}f\Vert^{2}.
\end{eqnarray*}
Moreover,
\begin{eqnarray*}
Re\left(\sum_{i\in J^{c}}\langle f, f_{i}\rangle \overline{\langle KK^{*}f, f_{i}\rangle} \right)&\leq& \left(\sum_{i\in I}\vert \langle f, f_{i}\rangle \vert^{2}\right)^{1/2}\left(\sum_{i\in I}\vert \langle KK^{*}f, f_{i}\rangle \vert^{2}\right)^{1/2}\\
&=& \Vert K^{*}f\Vert \Vert K^{*}KK^{*}f\Vert\\
&=& \Vert K^{\dagger}KK^{*}f\Vert \Vert K^{*}KK^{*}f\Vert\\
&\leq& \Vert K\Vert \Vert K^{\dagger}\Vert \Vert KK^{*}f\Vert^{2}.
\end{eqnarray*}
Hence,
\begin{eqnarray*}
v_{-}(F, K, J)\leq v_{+}(F, K, J)\leq \Vert K\Vert \Vert K^{\dagger}\Vert (1+\Vert K\Vert\Vert K^{\dagger}\Vert).
\end{eqnarray*}
 On the other hand, in the proof of  Theorem \ref{parseval1212} we observed that
\begin{eqnarray*}
 KK^{*}S_{J}-S_{J^{c}}KK^{*}=S_{J}^{2}-S_{J^{c}}^{2}.
\end{eqnarray*}
and consequently we have
\begin{eqnarray*}
\langle S_{J}^{2}f, f\rangle + \langle S_{J^{c}}KK^{*}f,f\rangle= \langle S_{J^{c}}^{2}f, f\rangle + \langle KK^{*}S_{J}f, f\rangle ,
\end{eqnarray*}
for every $f\in \mathcal{H}$. Thus
\begin{eqnarray*}
&&\left\Vert \sum_{i\in J}\langle f, f_{i}\rangle f_{i}\right\Vert^{2} + \overline{\sum_{i\in J^{c}}\langle f, f_{i}\rangle \overline{\langle KK^{*}f, f_{i}\rangle}}\\
&=&  \left\Vert \sum_{i\in J^{c}}\langle f, f_{i}\rangle f_{i}\right\Vert^{2} + \sum_{i\in J}\langle f, f_{i}\rangle \overline{\langle KK^{*}f, f_{i}\rangle}.
\end{eqnarray*}
This shows (ii). Finally (iii) is easy to check.
\end{proof}
Using the result above-mentioned, we stablish some equivalent results for Parseval $K$-frames.
\begin{cor}
Let  $\{f_{i}\}_{i\in I}$ be  a Parseval $K$-frame for $\mathcal{H}$. Then for every $J\subset I$ and $f\in \mathcal{H}$ the following conditions are equivalent.
\begin{itemize}
\item[(i)]
$v_{+}(F, K, J)=v_{-}(F, K, J)=1$.
\item[(ii)]
$\left \Vert \sum_{i\in J}\langle f, f_{i}\rangle f_{i}\right \Vert^{2}=Re\sum_{i\in J}\langle f, f_{i}\rangle \overline{\langle KK^{*}f, f_{i}\rangle}.$
\item[(iii)]
$\left \Vert \sum_{i\in J^{c}}\langle f, f_{i}\rangle f_{i}\right \Vert^{2}=Re\sum_{i\in J^{c}}\langle f, f_{i}\rangle \overline{\langle KK^{*}f, f_{i}\rangle}$.
\end{itemize}
\end{cor}
\begin{proof}
$(ii)\Leftrightarrow (iii)$ is clear. Also,
$(i)\Rightarrow (ii)$   holds by a direct computation. Now,
let $(ii)$ holds,  then
\begin{eqnarray*}
\sum_{i\in J^{c}}\overline{\langle f, f_{i}\rangle} \langle KK^{*}f, f_{i}\rangle+ \left \Vert \sum_{i\in J}\langle f, f_{i}\rangle f_{i}\right \Vert^{2}&=& \sum_{i\in I}\overline{\langle f, f_{i}\rangle} \langle KK^{*}f, f_{i}\rangle\\
&=& \langle KK^{*}f, S_{F}f\rangle\\
&=& \left \Vert KK^{*}f\right \Vert^{2}.
\end{eqnarray*}
i.e., $(i)$ holds.
Hence  $(i)\Leftrightarrow (iii)$ and similarly  $(i)\Leftrightarrow (iii)$.
\end{proof}
Finally, we obtain the following more general result.
\begin{cor}
Let  $\{f_{i}\}_{i\in I}$ be  a Parseval $K$-frame for $\mathcal{H}$. Then for every $J\subset I$ and $f\in \mathcal{H}$ the following conditions are equivalent.
\begin{itemize}
\item[(i)]
$\left \Vert \sum_{i\in J}\langle f, f_{i}\rangle f_{i}\right \Vert^{2}=\sum_{i\in J}\langle f, f_{i}\rangle \overline{\langle KK^{*}f, f_{i}\rangle}.$
\item[(ii)]
$\left \Vert \sum_{i\in J^{c}}\langle f, f_{i}\rangle f_{i}\right \Vert^{2}=\sum_{i\in J^{c}}\langle f, f_{i}\rangle \overline{\langle KK^{*}f, f_{i}\rangle}$.
\item[(iii)]
$S_{J} f\perp S_{J^{c}}f$.
\item[(iv)]
$f\perp S_{J^{c}}S_{J}f$.
\end{itemize}
\end{cor}
\begin{proof}
Since,
\begin{eqnarray*}
\Vert S_{J}f\Vert^{2}-\langle KK^{*}S_{J}f, f\rangle =
\Vert S_{J^{c}}f\Vert^{2}-\langle KK^{*}S_{J^{c}}f, f\rangle.
\end{eqnarray*}
Thus $(i)\Leftrightarrow (ii)$. Moreover $S_{J}$ and $S_{J^{c}}$ are positive, so we have
\begin{eqnarray*}
\langle S_{J^{c}}f, S_{J}f\rangle = \langle f,  S_{J^{c}}S_{J}f\rangle = \langle (KK^{*}S_{J}-S_{J}^{2})f, f\rangle, \quad (f\in \mathcal{H}),
\end{eqnarray*}
This follows that, $ (iii)\Leftrightarrow (iv)$ and $(i)\Leftrightarrow (iv)$.
\end{proof}

\section{Some new inequalities}
In this section, applying Jensen's operator inequality  we obtain some new inequalities for $K$-frames.
First, recall that a continuous function $h$ defined on an interval $J$ is said to be operator convex if
\begin{eqnarray*}
h(\lambda A+(1-\lambda)B) \leq \lambda h(A) + (1-\lambda)h(B),
\end{eqnarray*}
for all $\lambda\in [0,1]$ and all self-adjoint operators $A$,$B$ with spectra in $J$.  A general formulation of Jensen's operator inequality is given as follows (see \cite{1}).
\begin{thm}\label{jensen1}
Let $\mathcal{H}$, $\mathcal{K}$ be Hilbert spaces and $h:J\rightarrow \mathbb{R}$ be an operator convex function. For each $n\in \mathbb{N}$, the inequality
\begin{eqnarray}\label{jensen}
h(\sum_{i=1}^{n}\Phi_{i}(A_{i})) \leq \sum_{i=1}^{n}\Phi_{i}(h(A_{i})),
\end{eqnarray}
holds for every $n$-tuple $(A_{1},..., A_{n})$ of self-adjoint operators in $B(\mathcal{H})$  with spectra in $J$ and every $n$-tuple $(\Phi_{1},..., \Phi_{n})$ of positive linear mappings $\Phi_{i}:B(\mathcal{H})\rightarrow B(\mathcal{K})$ such that $\sum_{i=1}^{n}\Phi_{i}(I_{\mathcal{H}})=I_{\mathcal{K}}$.
\end{thm}
Also, a variant of Jensen's operator inequality (\ref{jensen}) for convex functions is presented as follows (see \cite{2}).
\begin{thm}\label{jensen's20}
Let $A_{1}, ..., A_{n}$ be self-adjoin operators with spectra in $[m, M]$ for some scalars $m<M$ and $\Phi_{1},..., \Phi_{n}$ be positive linear mappings from $B(\mathcal{H})$ into $B(\mathcal{K})$ with $\sum_{i=1}^{n}\Phi_{i}(I_{\mathcal{H}})=I_{\mathcal{K}}$, where $I_{\mathcal{H}}$ and $I_{\mathcal{K}}$ are the identity mappings on $\mathcal{H}$ and $\mathcal{K}$, respectively. If $h:[m,M]\rightarrow \mathbb{R}$ is a continuous convex function, then
\begin{eqnarray}\label{jensen2}
h\left(mI_{\mathcal{K}}+MI_{\mathcal{K}}-\sum_{i=1}^{n}\Phi_{i}(A_{i})\right)\leq h(m)I_{\mathcal{K}}+h(M)I_{\mathcal{K}}-\sum_{i=1}^{n}\Phi_{i}(h(A_{i})).
\end{eqnarray}
\end{thm}
Now, we apply inequalities $(\ref{jensen})$ and $(\ref{jensen2})$ to obtain some inequalities for $K$-frames.
\begin{thm}\label{jensens, k-frame}
Suppose $F=\{f_{i}\}_{i\in I}$ is a Parseval $K$-frame for $\mathcal{H}$. Then for every subset $J$ of $I$,
\begin{itemize}
\item[(i)]
if $h:[0, \Vert K\Vert^{2}]\rightarrow \Bbb{R}$ is an operator convex function, then
 \begin{eqnarray}\label{pre1}
 h\left(\dfrac{KK^{*}}{2}\right) \leq \dfrac{h(S_{J})+h(S_{J^{c}})}{2};
 \end{eqnarray}
\item[(ii)]
if $h:[0, \Vert K\Vert^{2}]\rightarrow \Bbb{R}$ is a  convex function, then
 \begin{eqnarray}\label{pre2}
 h\left(\Vert K\Vert^{2}I_{\mathcal{H}}-\dfrac{KK^{*}}{2}\right) \leq  h(\Vert K\Vert^{2}I_{\mathcal{H}})-\dfrac{h(S_{J})+h({S_{J^{c}}})}{2}.
 \end{eqnarray}
\end{itemize}
\end{thm}
\begin{proof}
Let $A_{1}=S_{J}$ and $A_{2}=S_{J^{c}}$. Since  $S_{J}$ and $S_{J^{c}}$ are positive operators with  $S_{J}+S_{J^{c}}=KK^{*}$, the spectra of $A_{1}$ and $A_{2}$ are in $[0,\Vert K\Vert^{2}]$. Now, if we define $\Phi_{1}=\dfrac{1}{2}Id$ and $\Phi_{2}=\dfrac{1}{2}Id$, where $Id$ is the identity mapping on $B(\mathcal{H})$, then $(i)$ and  $(ii)$ follow from inequalities $(\ref{jensen})$ and $(\ref{jensen2})$, respectively.
\end{proof}
\begin{cor}
Suppose $F=\{f_{i}\}_{i\in I}$ is a Parseval $K$-frame for $\mathcal{H}$. Then for every subset $J$ of $I$ and $f\in \mathcal{H}$, we obtain
\begin{eqnarray}\label{cor1}
\nonumber  \dfrac{1}{2}\Vert KK^{*}f\Vert^{2} &\leq & \left\Vert \sum_{i\in J} \langle f, f_{i}\rangle f_{i}\right\Vert^{2}+\left\Vert \sum_{i\in J^{c}} \langle f, f_{i}\rangle f_{i}\right\Vert^{2} \\
   &\leq & 2\Vert K\Vert^{2}\Vert K^{*}f\Vert^{2}-\dfrac{1}{2}\Vert KK^{*}f\Vert^{2}.
\end{eqnarray}
\end{cor}

\begin{proof}
Since the function $f(x)=x^2$ is convex on $[0,\infty)$, by Theorem $\ref{jensens, k-frame}$ we have
\begin{eqnarray}\label{proofcor1}
\dfrac{(KK^*)^2}{4}\leq \dfrac{S_J^2+S_{J^c}^2}{2},
\end{eqnarray}
and
$$\left(\Vert K\Vert^2I_\mathcal{H}-\dfrac{KK^*}{2}\right)^2\leq \Vert K\Vert^4I_\mathcal{H}-\dfrac{S_J^2+S_{J^c}^2}{2},$$
and so
\begin{eqnarray}\label{proofcor2}
\dfrac{S_J^2+S_{J^c}^2}{2}\leq  \Vert K\Vert ^2KK^*-\dfrac{(KK^*)^2}{4}.
\end{eqnarray}
Therefore, it follows from $(\ref{proofcor1})$ and $(\ref{proofcor2})$ that
$$\dfrac{(KK^*)^2}{2}\leq S_J^2+S_{J^c}^2 \leq 2\Vert K\Vert ^2KK^*-\dfrac{(KK^*)^2}{2}.$$
Hence for every $f\in \mathcal{H}$, we obtain inequality $(\ref{cor1})$.
\end{proof}




\bibliographystyle{amsplain}

\begin{thebibliography}{10}


\bibitem{arefi3} F. Arabyani Neyshaburi, A. Arefijamaal, Some constructions of $K$-frames and their duals, To appear in Rocky Mountain. Math.



\bibitem{Balan} R. Balan, P. G. Casazza, D. Edidin and  G. Kutyniok,  A new identity for Parseval frames, Proc. Amer. Math. Soc.
\textbf{135} (2007), 1007--1015.

\bibitem{Balan2} R. Balan, P. G. Casazza and  D. Edidin,  On signal reconstruction without phase, Appl. Comput. Harmon. Anal.
\textbf{20} (2006), 345--356.



\bibitem{Ben06} J. Benedetto, A.  Powell  and  O. Yilmaz,   Sigm-Delta quantization and finite frames,  IEEE Trans. Inform. Theory. \textbf{52} (2006), 1990--2005.



\bibitem{Bod05}  B. G. Bodmannand,  V. I. Paulsen,  Frames, graphs and erasures,
Linear. Algebra  Appl. {\bf 404} (2005), 118--146.

\bibitem{Bol98}  H. Bolcskel, F.  Hlawatsch  and H. G.  Feichtinger,  Frame-theoretic analysis of oversampled filter banks, IEEE Trans. Signal Process. \textbf{46} (1998), 3256--3268.




\bibitem{Cas00} P. G. Casazza, The art of frame theory, Taiwanese J. Math. {\bf 4}(2) (2000), 129-202.






\bibitem{Chr08} O. Christensen, Frames and Bases: An Introductory Course,
Birkh\"{a}user, Boston. 2008.






\bibitem {Douglas} R. G. Douglas, On majorization, factorization and range inclusion of operators on Hilbert space, Proc. Amer. Math. Soc. {\bf 17}(2) (1966), 413-415.









\bibitem {Han} H. G. Feichtinger, T. Werther, Atomic systems for subspaces, in: L. Zayed(Ed.), proceedings SampTA. Orlando, FL, (2001), 163-165.


\bibitem {Gav06} P. G$\breve{\textrm{a}}$vru\c{t}a, On some identities and inequalities for frames in Hilbert spaces, J. Math. Anal. Appl. {\bf 321} (2006), 467-478.

\bibitem{1} F. Hansen, J. Pe$\check{c}$ari$\acute{c}$, I. Peri$\acute{c}$,
  Jensen�s operator inequality and it�s converses,
  Math. Scand. {\bf 100} (2007), 61--73.

\bibitem {Gav07} L. G$\breve{\textrm{a}}$vru\c{t}a, Frames for operators, Appl. Comput. Harmon. Anal. {\bf 32} (2012), 139-144.
\bibitem {Hei15} S. B. Heineken, P. M. Morillas, Properties of finite dual fusion frames,
Linear Algebra Appl. \textbf{453} (2014), 1-27.








\bibitem{2} A. Matkovi$\acute{c}$, J. Pe$\check{c}$ari$\acute{c}$, I. Peri$\acute{c}$, A variant of Jensen's inequality of Mercer's type for operators with applications, Linear Algebra Appl. {\bf 418} (2006), 551--564.

\bibitem{Paw} M. Pawlak, U. Stadtmuller, Recovering band-limited signals under noise, IEEE Trans. Info. Theory. \textbf{42} (1994), 1425-1438.










\bibitem{Werthrt} T. Werther, Reconstruction from irregular samples with improved locality, Masters thesis, University of Vienna, Dec. 1999.
\bibitem{Xiao} X. C. Xiao, Y. C. Zhu and L. G$\breve{\textrm{a}}$vru\c{t}a. Some properties of K-frames in Hilbert spaces, Results. Math.  \textbf{63} (2013), 1243-1255.

\bibitem{Xiao2} X. Xiao, Y. Zhu and M. Ding, Erasures and  equalities for fusion frames in Hilbert spaces, Bull. Malays. Math. Sci. Soc.  \textbf{2014} (2014), 1--11.

\bibitem{Zhu} X. Zhu,  G. Wu, A note on  some equalities for  frames in Hilbert spaces, Appl. Math. Lett.  \textbf{23}(7) (2010), 788-790.


\end{thebibliography}

\end{document}